\documentclass[12pt]{amsart}
\usepackage[english]{babel}
\title{Mild parametrizations of power-subanalytic sets}
\author[S. Van Hille]{Siegfried Van Hille}
\address{KU Leuven, Celestijnenlaan 200B, 3001 Leu\-ven, Bel\-gium \newline Fields Institute, 222 College Street, Toronto M5T 3J1, Ontario, Canada}
\email{siegfried.vanhille@kuleuven.be \newline svanhill@fields.utoronto.ca}

\usepackage[margin = 3cm]{geometry}
\usepackage{amsfonts,amsmath,amsthm,amssymb,amsopn}
\usepackage{enumitem}
\usepackage{parskip}
\usepackage[colorlinks]{hyperref}
\hypersetup{
	linkcolor = {red},
	citecolor = {blue},
	urlcolor = {blue},
}
\urladdr{\href{https://sites.google.com/view/siegfriedvanhille/academic}{https://sites.google.com/view/siegfriedvanhille/academic}}

\DeclareMathOperator{\N}{\mathbb N}

\DeclareMathOperator{\Q}{\mathbb Q}
\DeclareMathOperator{\R}{\mathbb R}

\DeclareMathOperator{\im}{\text{Im}}
\DeclareMathOperator{\F}{\mathcal F}
\newcommand{\RKF}{\mathbb R^K_{\mathcal F}}
\newcommand{\Ran}{\mathbb R_{\text{an}}}
\newcommand{\RanR}{\mathbb R_{\text{an}}^{\mathbb R}}
\newcommand{\phinf}{\phi^\infty}
\newcommand{\cell}{\mathfrak C}

\theoremstyle{plain}
\newtheorem{theorem}{Theorem}[section]
\newtheorem{corollary}[theorem]{Corollary}
\newtheorem{lemma}[theorem]{Lemma}
\newtheorem{proposition}[theorem]{Proposition}
\newtheorem*{theorem*}{Theorem}
\newtheorem*{corollary*}{Corollary}
\newtheorem*{proposition*}{Proposition}
\newtheorem*{conj*}{Conjecture}

\theoremstyle{definition}
\newtheorem{definition}[theorem]{Definition}
\newtheorem*{conv*}{Convention}
\newtheorem{remark}[theorem]{Remark}

\makeatletter
\@namedef{subjclassname@2020}{%
  \textup{2020} Mathematics Subject Classification}
\makeatother

\begin{document}
\thanks{The author is partially supported by KU Leuven Grant IF C14/17/83.}

\begin{abstract}
We obtain two uniform parametrization theorems for families of bounded sets definable in $\RanR$. Let $X = \{X_t \subset (0,1)^n \mid t \in T\}$ be a definable family of sets $X_t$ of dimension at most $m$. Firstly, $X_t$ admits a $C^r$-parametrization consisting of $cr^m$ maps for some positive constant $c = c(X)$, which is uniform in $t$. Secondly, $X_t$ admits a $C$-mild parametrization for any $C>1$, which is also uniform in $t$.
\end{abstract}

\maketitle

\section{Introduction}\label{secintro}
A parametrization of a set $X \subseteq \R^n$ is a finite set of maps $(0,1)^m \to X$, where $m$ is the dimension of $X$, whose ranges together cover $X$. There are different types of parametrizations, depending on the properties that are imposed on these maps. In this paper, we study parametrizations with mild maps up to order $r$, where the maps are $r$ times continuously differentiable and there is a growth condition on these derivatives depending on the order. If $r = +\infty$, the maps are mild as defined by Pila in \cite{milddef}, while if $r \in \N$, these functions were introduced in \cite{unif}. See Definition \ref{defmild} for a precise definition of mild functions.

The two main results of this paper are the following parametrization results for power-subanalytic sets (see Section \ref{secomin} for a definition of power-subanalytic). Firstly, for each $r \in \N$ and power-subanalytic set $X \subseteq (0,1)^n$, there exists a parametrization of $X$ where the maps are mild up to order $r$. By a linear reparametrization of $(0,1)^m$, this result implies that $X$ has a $C^r$-parametrization consisting of $cr^m$ maps for some $c>0$ (see Corollary \ref{corcrpara}). Secondly, for each $C>1$, $X$ has a $C$-mild parametrization, where the maps are $C$-mild. Both results are uniform in the sense that if $X = \{X_t \mid t \in T\}$, where $T$ is some power-subanalytic set of parameters, then the number of maps only depends on $X$, i.e., the number is uniform in $t$.

Parametrizations were first constructed by Yomdin and Gromov in the study of entropy of dynamical systems \cite{entropyadd,entropy,gromov}. They constructed $C^r$-parametrizations for semi-algebraic sets. This allowed Yomdin to give a proof of Shub's Entropy Conjecture for $C^\infty$ maps in \cite{entropy}. The techniques of Yomdin and Gromov were generalized by Pila and Wilkie in \cite{countthm} to sets definable in an o-minimal structure. In this way the results obtained by Bombieri and Pila in \cite{determinant} could be generalized to what is now well known as the Counting Theorem (see Theorem \ref{thmcount}).

In order to improve these applications, it turns out to be interesting to know how many charts are needed, in terms of $r$ and the combinatorial data defining $X$, i.e. $m,n$ and, if $X$ is semi-algebraic, the complexity of $X$. Strong results in this fashion were recently obtained by Binyamini and Novikov in \cite{ccs}. They show that any subanalytic set has a $C^r$-parametrization of $cr^m$ charts for some $c>0$. Moreover, if $X$ is semi-algebraic, it is shown that $c$ depends polynomially on the complexity of $X$. Building on work of Burguet, Liao and Yang \cite{entropy2}, Yomdin proves in Appendix A of \cite{ccs} a conjecture of himself on the entropy of analytic maps in arbitrary dimension, which he and already proved in dimension two \cite{entropyanalytic}. In Appendix B, the authors of \cite{ccs} improve some results of Pila and Wilkie as well.

Another recent result has been obtained by Cluckers, Pila and Wilkie in \cite{unif}. They show that any power-subanalytic set $X$ has a $C^r$-parametrization, where the number of maps depends polynomially on $r$, uniform for families. They were mainly interested in the diophantine applications and obtain some similar results as in \cite{ccs}, but in the larger class of power-subananalytic sets. In view of the above, it seems natural to make this polynomial in $r$ more explicit in terms of the combinatorial data defining $X$. I have shown in \cite{smoothpara} that the polynomial has degree $m^3$ by precisely analyzing their method. In this paper, I improve their construction such that the degree can be taken to be $m$ (as in \cite{ccs}, but for the more general class of power-subanalytic sets), see Theorem \ref{thmcrpara}. In this way, I can improve one of the diophantine applications of \cite{unif}, see Corollary \ref{corratpoints}.

Parametrizations with mild maps were considered by Pila in \cite{milddef} in order to achieve an improvement (Wilkie's Conjecture, see Section \ref{secappl}) of the results in \cite{countthm} for Pfaffian curves. Such a parametrization, a mild parametrization, is harder to establish. For instance, if $X$ is definable in an o-minimal structure, it does not necessarily follows that $X$ has a mild parametrization  (see \cite{nomildthomas}). There are also issues with uniformity. In particular, Yomdin showed in \cite{counterex} that the semi-algebraic family of hyperbolas defined by $\{ (x,y) \in (0,1)^2 \mid xy = t \}$ does not have a parametrization with charts that are $0$-mild, where the number of maps is independent of $t$. I have recently shown in \cite{mildpara} that it has a $C$-mild parametrization for any $C>0$, uniform in $t$. The same holds for any power-subanalytic family of curves. However, it seems that this result cannot be generalized to higher dimensions, where we only achieve $C>1$  (see Section \ref{secdiscuss}).

By work of Jones, Miller and Thomas \cite{nulmild}, we know that any subanalytic set has a $0$-mild parametrization, but this result is not uniform. Finally, in \cite{ccs}, it is also shown that any subanalytic set has a $2$-mild parametrization, which is uniform. In this paper we extend this result to power-subanalytic sets and $C$-mild for each $C>1$. An important difference between the results here (and in \cite{ccs}) on mild parametrizations and the result of \cite{nulmild}, is that here, the maps of the parametrization are not definable within the same o-minimal structure, while this is the case in \cite{nulmild}. We refer to \cite{deryathesis} for more details on definability questions of mild parametrizations in o-minimal structures, and mild parametrization in other, larger, o-minimal structures.

The outline of the paper is as follows. In Section \ref{secprelim} some preliminary definitions and properties are recalled and the main theorems are precisely formulated. Most important is the so-called Pre-parametrization Theorem, which is the foundation of the proofs of the main results. Section \ref{secinterm} contains some key lemmas that will be used in the proofs these results. It also contains a modified proof of \cite[Proposition 4.1.5]{unif}, that is intended to make these proofs more clear. Finally, in Section \ref{secmain}, the main theorems are proved. I also discuss overlap and differences with earlier work, and the diophantine application.

\section{Preliminaries}\label{secprelim}
We introduce the notation of multivariate calculus, recall some important definitions and state the Pre-parametrization Theorem that will be crucial for the proofs of the main theorems.

\subsection{Multivariate calculus}
Let $r \in \N \cup \{+\infty\}$. A function $f: U \to \R^n$, where $U$ is an open subset of $\R^m$, is $C^r$ if all of its component functions are $r$ times continuously differentiable and for every $\nu \in \N^m$, with $|\nu| = \nu_1+\ldots+\nu_m \leq r$, we denote
\[
f^{(\nu)} = \left( \frac{\partial^{|\nu|}}{\partial x_1^{\nu_1} \cdots \partial x_m^{\nu_m}}f_1,\ldots, \frac{\partial^{|\nu|}}{\partial x_1^{\nu_1} \cdots \partial x_m^{\nu_m}}f_n \right).
\]
Furthermore, $\nu! = \nu_1!\cdots\nu_m!$, for $\mu \in \R^m$, we denote $x^\mu = x^{\mu_1}\cdots x^{\mu_m}$, and by definition we set $0^0 = 1$.

We will be dealing with arbitrary derivatives of compositions of functions. There is a formula for that, say a general version of the chain rule, which is known as the Fa\`a di Bruno formula.

\begin{proposition}[Fa\`a di Bruno, \cite{faa}]\label{faa}
Suppose that $r$ is a positive integer, $V \subseteq \R^d$ and $U \subseteq \R^e$ are open, $f: V \to \R$, $g: U \to V$ and that $f$ and $g$ are $C^r$. For any $x \in U$ and $\nu \in \N^e$ with $|\nu| \leq r$ we have that:
\[
(f \circ g)^{(\nu)}(x) = \sum_{1 \leq |\lambda| \leq |\nu|} f^{(\lambda)}(g(x)) \sum_{s = 1}^{|\nu|} \sum_{p_s(\nu,\lambda)} \nu! \prod_{j = 1}^s \frac{(g^{(l_j)}(x))^{k_j}}{k_j!(l_j!)^{|k_j|}},
\]
where $p_s(\nu,\lambda)$ is the set consisting of all $k_1,\ldots,k_s \in \N^d$ with $|k_i| > 0$ and $l_1,\ldots,l_s \in \N^e$ with $0 \prec l_1 \prec \ldots \prec l_s$ such that: $$ \sum_{i = 1}^sk_i = \lambda$$ and $$\sum_{i = 1}^s |k_i|l_i = \nu.$$ Here $l_i \prec l_{i+1}$ means that $|l_i| < |l_{i+1}|$ or, if $|l_i| = |l_{i+1}|$, then $l_i$ comes lexicographically before $l_{i+1}$.
\end{proposition}

We will almost always use the case $d = e = m$, where $g$ will be a reparametrization of the domain of $f$, and $m$ is the dimension of the family that is parametrized. Finally, we will use the following norm. 

\begin{definition}[$C^r$-norm]\label{defcrnorm}
Let $r \in \N$ and $f: U \to \R$ be $C^r$. Then we define $C^r$-norm by
\[
|f|_r = \sup_{\substack{x \in U \\ \nu \in \N^m : |\nu| \leq r}} \frac{\left| f^{(\nu)}(x) \right|}{|\nu|!}.
\]
For a $C^r$ map $f: U \to \R^n$, we define $|f|_r$ to be the maximum of the $C^r$-norms of its component functions.
\end{definition}

\subsection{Mild functions}
Mild functions were introduced by Pila in \cite{milddef} in order to bound the number of rational points on a Pfaffian curve. We will use the definition of \cite{smoothpara}, where I derived a result on the composition of such functions by translating them to Gevrey functions (see \cite{gevrey}), and using the results that are known for these functions.

\begin{definition}[Mild functions]\label{defmild}
Let $r \in \N \cup \{+\infty\}$ and $A,B > 0$ and $C \geq 0$ be real numbers. A function $f: U \subseteq \R^m \to \R$, where $U$ is open, is $(A,B,C)$-mild up to order $r$ if it is $C^r$ on $U$ and if for any $\nu \in \N^m$ with $|\nu| \leq r$ and $x \in U$:
\[
\left| f^{(\nu)}(x) \right| \leq BA^{|\nu|} |\nu|!^{1+C}.
\]
A map $f: U \to \R^n$ is $(A,B,C)$-mild up to order $r$ if all of its component functions are. We may simply write $C$-mild up to order $r$, which means $(A,B,C)$-mild up to order $r$ for some $A,B>0$, or just mild up to order $r$, which means $(A,B,C)$-mild up to order $r$ for some $A,B > 0$ and $C \geq 0$. Also, if $r = +\infty$, we may often omit ``up to order $+\infty$''.
\end{definition}

Several examples will appear in the paper, but let us already mention that if $f: U \to \R$, where $U$ is a bounded subset of $\R^n$, is analytic on an open neighbourhood of the topological closure of $U$, then it is $0$-mild (see \cite[Proposition 2.2.10]{primer}). 

Let us now give some elementary properties of mild functions. The second and third result can be found in \cite{smoothpara}, the first one is clear. 

\begin{proposition}[{ \cite{smoothpara} }]\label{propmildop}
Suppose that $f,g: U \subseteq \R^m \to \R$ are $(A,B,C)$-mild up to order $r$.
\begin{enumerate}
\item The function $f+g$ is $(A,2B,C)$-mild up to order $r$.
\item The function $f\cdot g$ is $(2A,B^2,C)$-mild up to order $r$.
\end{enumerate}
Suppose $f: V \subseteq \R^m \to \R$ is $(A_f,B_f,C)$-mild up to order $r$ and $g: U \subseteq \R^m \to C$ is $(A_g,B_g,C)$-mild up to order $r$.
\begin{enumerate}[resume]
\item The composition $f \circ g$ is $(\tilde{A}, B_f,C)$-mild up to order $r$, with
\[
\tilde{A} = A_g(mB_gA_f+1)^{1+C}.
\]
\end{enumerate}
\end{proposition}

\begin{remark}
If $f$ is $(A_f,B_f,C_f)$-mild up to order $r_f$ and $g$ is $(A_g,B_g,C_g)$-mild up to order $r_g$, then one can first take the maximum of the mildness parameters and take the minimum of $r_f$ and $r_g$ in order to apply the above proposition.
\end{remark}

\subsection{Main results}
Our main results hold for power-subanalytic families of subsets of $(0,1)^n$. We define the notion power-subanalytic in the next subsection. We will denote a family of such subsets by $X = \{X_t \subseteq (0,1)^n \mid t \in T\}$, where $T \subseteq \R^k$ is some (power-subanalytic) set of parameters. The dimension of $X$ is the maximum of the dimensions of its family members. We will work with (power-subanalytic) families of maps $f: U \subseteq T \times (0,1)^m \to \R^n$ and for any $t \in T$, denote $f_t$ for the map $U_t \to \R^n$ defined by $f_t(x) = f(t,x)$, where $U_t = \{ x \in (0,1)^m \mid (t,x) \in U\}$.

The following parametrization theorem is the first main result of this paper. It further refines the main result of \cite{smoothpara}, which made the polynomial dependence on $r$, obtained by Cluckers, Pila and Wilkie in \cite{unif}, more explicit. More precisely, we improve the degree of $r$ from $m^3$ to $m$. This is the same number of charts as in \cite{ccs} for families of subanalytic sets.

\begin{theorem*}[Theorem \ref{thmcrpara}]
Let $X$ be a power-subanalytic family of subsets of $(0,1)^n$ of dimension $m$. Then there exist constants $c = c(X) \in \N$ and $A = A(X) \in \N$ with the following property. For all $r \in \N$, there exist power-subanalytic maps $f_l: T \times (0,1)^m \to X$, for $1 \leq l \leq c$, such that for any $t \in T$, we have that
\[
\bigcup_{l=1}^c \im(f_{l,t}) = X_t,
\]
and for each $l \in \{1,\ldots,c\}$, the map $f_{l,t}$ is $(Ar,1,0)$-mild up to order $r$.
\end{theorem*}

Using a linear reparametrization, one obtains the following $C^r$-parametrization theorem of $X$.

\begin{corollary}[$C^r$-parametrization]\label{corcrpara}
Let $X$ be a power-subanalytic family of subsets of $(0,1)^n$ of dimension $m$. Then there exists a constant $c = c(X) \in \N$ with the following property. For all $r \in \N$, there exist power-subanalytic maps $f_l: T \times (0,1)^m \to X$, for $1 \leq l \leq cr^m$, such that for any $t \in T$, we have that
\[
\bigcup_{l = 1}^{cr^m} \im(f_{l,t}) = X_t,
\]
and for each $l \in \{1,\ldots,cr^m\}$, $|f_{l,t}|_r \leq 1$.
\end{corollary}

\begin{remark}\label{remamountcharts}
The norm plays a role in the number of maps in the corollary above. If one would use that $C^r$ supremum norm, defined by
\[
\sup_{\substack{x \in U \\ \nu \in \N^m : |\nu| \leq r}} \left| f^{(\nu)}(x) \right|,
\]
then the number of maps obtained by applying a linear reparametrization to the maps obtained from Theorem \ref{thmcrpara} is $cr^{2m}$ for some $c>0$ (the same $c$ as in the corollary).
\end{remark}

It is encouraged to do either Corollary \ref{corcrpara} or Remark \ref{remamountcharts} as an exercise.

\begin{remark}
In \cite{ccs} one also uses the $C^r$-norm defined in Definition \ref{defcrnorm}, and obtain $cr^m$ $C^r$-charts for subanalytic sets. In \cite{unif}, one uses the $C^r$ supremum norm defined above.
\end{remark}

The second main result of this paper is a $C$-mild parametrization theorem, for $C>1$, that follows from earlier work in \cite{mildpara}. If one picks $C = 2$, then one obtains the uniform $(A,1,2)$-mild parametrization result for subanalytic families of \cite{ccs}. However, if one restrict to the case of curves, then it is possible to take $C>0$. In higher dimensions, I do not know whether this sharper bound would follow from the methods in this paper.

\begin{theorem*}[Mild parametrization, Theorem \ref{thmmildpara}]
Let $X$ be a power-subanalytic family of subsets of $(0,1)^n$ of dimension $m$. Then there exist constants $c = c(X) \in \N$ and $A = A(X) \in \N$ with the following property. For all $C>1$, there exist maps $f_l: T \times (0,1)^m \to X$, for $1 \leq l \leq c$, such that for any $t \in T$, we have
\[
\bigcup_{l=1}^c \im(f_{l,t}) = X_t,
\]
and for each $l \in \{1,\ldots,c\}$, the map $f_{l,t}$ is $(A/C,1,C)$-mild up to order $+\infty$.
\end{theorem*}

Using a linear reparametrization, one could also derive a $C^r$-parametrization theorem from this result. If $C = 1$, the number of maps obtained is, up to some constant, the same as the number obtained in Corollary \ref{corcrpara}, so we reach this ``asymptotically''.

Both results are deduced from a Pre-parametrization Theorem, Theorem \ref{prepara}, which we will formulate below. This theorem allows us to uniformly parametrize $X$ with finitely many maps $f$ of a simple form. One applies a power substitution to these maps, roughly speaking, raise each variable to the power $r$, to obtain Theorem \ref{thmcrpara}. Applying an exponential substitution to these maps yields Theorem \ref{thmmildpara}. 

\subsection{Results and terminology of o-minimality} \label{secomin}
A subset of $\R^n$ or a function is \textbf{power-subanalytic} if it is definable in the o-minimal structure $\RKF$, where $\F$ is a Weierstrass system, and $K$ a subfield of the field of exponents of $\F$. See \cite{preparation} for a precise definition of these structures and of a Weierstrass system. It is not important for our results to precisely understand these structures. Let us mention that for the smallest Weierstrass system $\F$ and $K = \Q$, this is the structure of real semi-algebraic sets, and for the largest choice of $\F$ and $K \subseteq \R$, the collection of all \emph{restricted} analytic functions, this is the expansion of the structure $\Ran$ of globally subanalytic sets by power functions $x \mapsto x^r$ for $x>0$ and $r \in K$.

The domain of the functions we will encounter will be open \textbf{cells} in $(0,1)^m$. An \emph{open} cell $\cell$ in $(0,1)^m$ is a set defined by inequalities of the form
\[
\alpha_i(x_1,\ldots,x_{i-1}) < x_i < \beta_i(x_1,\ldots,x_{i-1}),
\]
for some continuous power-subanalytic functions $\alpha_i,\beta_i$ with $\alpha_i < \beta_i$ for $i = 1,\ldots,m$. The functions $\alpha_i$ and $\beta_i$ are called the walls of the variable $x_i$ and the walls of all variables are also called the walls of $\cell$. By o-minimality, every power-subanalytic subset of $\R^n$ is a finite union of cells. Note that some of these cells might not be open (in that case one has an equality for some variable(s), instead of inequalities).

This is all terminology of o-minimality that is required to understand this paper. Theorem \ref{prepara} below heavily relies on this theory, but we will just use it as a black box. First we need two more definitions.

\begin{definition}[Bounded monomial] \label{defbm}
A function $b: U \subseteq T \times (0,1)^m \to \R$ is called a bounded monomial map if it has bounded range, and is of the form
\[
b(t,x) = a(t)x^\mu
\]
for some power-subanalytic function $a$ and $\mu \in \R^m$. A map $b: U \subseteq T \times (0,1)^m \to \R^n$ is bounded monomial if each of its component functions are.
\end{definition}

\begin{remark} \label{remdefbm}
Let $b: U \subseteq T \times (0,1)^m \to \R^n$ be a bounded monomial map. It follows from the definition that for every $t \in T$, the function $b_t: U_t \to \R$ is bounded. Moreover, this bound may be taken independent from $t$.
\end{remark}

\begin{definition}[Prepared in $x$] \label{defprepx}
A function $f: U \subseteq T \times (0,1)^m \to \R$ is prepared in $x$ if there exists a bounded monomial map $b: U \to \R^N$ $(N \in \N)$, and an analytic function $F$ that is non-vanishing on an open neighbourhood of the topological closure of $\im(b)$, such that
\[
f(t,x) = b_j(t,x)F(b(t,x)),
\]
where $b_j$ is a component function of the map $b$. The map $b$ is called the associated bounded monomial map of $f$.
\end{definition}

\begin{remark}\label{remdefprepx}
Note that it follows from the definition that the map $F$, often called the unit, is $(A_F,B_F,0)$-mild for some $A_F,B_F$ that depend on $F$ only. In particular, $A_F$ and $B_F$ do not depend on $t \in T$.
\end{remark}

We now state the Pre-parametrization Theorem that is the key result in the proofs of Theorem \ref{thmcrpara} and Theorem \ref{thmmildpara}. It is the version from \cite{smoothpara}, but it is originally due to Cluckers, Pila and Wilkie in \cite{unif}.

\begin{theorem}[Pre-parametrization, {\cite[Theorem 3.11]{smoothpara}}]\label{prepara}
Let $T \subseteq \R^k$ and $X = \{X_t \mid t \in T\}$ be a power-subanalytic family of subsets of $(0,1)^n$ of dimension $m$. Then there exist finitely many power-subanalytic families of maps $f_l: \cell_l \to (0,1)^n$ such that:
\begin{enumerate}
\item for each $t \in T$: $\bigcup_{l} \im(f_{l,t}) = X_t$;
\item for each $l$, $\cell_l$ is a cell contained in $T_l \times (0,1)^m$, where $T_l$ is a cell contained in $T$, and for each $t \in T$, $\cell_{l,t}$ is an open cell in $(0,1)^m$;
\item for each $l$, $f_l$ is prepared in $x$ and for its associated bounded monomial map $b_l$, there is a $B>0$ such that for any $t \in T$, the $C^1$-norm of $b_{l,t}$ is at most $B$;
\item for each $l$, the walls of $\cell_l$ bounding the variables $x_1,\ldots,x_m$ (extended trivially to functions in $(t,x)$) are of the same form as $f_l$, as defined in property 3.
\end{enumerate}
\end{theorem}

\begin{remark}
Note that $(3)$ is equivalent to stating that the associated bounded monomial map $b$, of some map $f$ that is prepared in $x$, has bounded $C^1$-norm, where only differentiation with respect to the $x$ is taken into account. Coupled with remarks \ref{remdefbm} and \ref{remdefprepx}, properties $(3)$ and $(4)$ of Theorem \ref{prepara} ensure the uniformity of the results.
\end{remark}

\section{Intermediate results} \label{secinterm}
In this section we prove as an illustration a power substitution result for functions obtained by Theorem \ref{prepara}. It is not new, but we present in a slightly different way, such that the exposition of the proof of Theorem \ref{thmcrpara} is more clear. Before stating and proving this proposition, we need some lemmas.

\begin{lemma}\label{lemmaAB}
Let $A_1,A_2,B_1,B_2$ be strictly positive real numbers. Then we have
\[
\sum_{1 \leq |\lambda| \leq |\nu|} B_1A_1^{|\lambda|} |\lambda|! \sum_{s = 1}^{|\nu|} \sum_{p_s(\nu,\lambda)} \nu! \prod_{j=1}^s \frac{(B_2 A_2^{|l_j|} |l_j|!)^{|k_j|}}{k_j!(l_j)!^{|k_j|}} = \frac{mA_1B_1B_2}{mA_1B_2+1}(A_2(mA_1B_2+1))^{|\nu|}|\nu|!,
\]
where the sum is subject to the same conditions as described in Proposition \ref{faa}.
\end{lemma}
\begin{proof}
This is essentially what is shown in the proof of \cite[Theorem 2.10]{smoothpara}, which is deduced from Gevrey's work \cite{gevrey}.
\end{proof}

In particular, this formula implies the result on compositions of mild functions in Proposition \ref{propmildop}. Together with the Fa\`a di Bruno formula, this lemma will be the key to all our estimates. In the following two lemmas we study the derivatives of functions that are prepared in $x$.

\begin{lemma}\label{lemmabmmap}
Let $b: U \subseteq T \times (0,1)^m \to \R$, where for each $t \in T$, $U_t$ is an open subset of $(0,1)^m$, be a bounded monomial map. Then for any $\nu \in \N^m$, $t \in T$ and $x \in U_t$, we have that
\[
\left| b_t^{(\nu)}(x) \right| \leq x^{-\nu} |b_t(x)| M_b^{|\nu|} |\nu|!,
\]
where $M_b = \max(|\mu_1|, \ldots, |\mu_m|,1)$.
\end{lemma}
\begin{proof}
Let $t \in T$ and $x \in U_t$. By the form of $b_t$, it is clear that for any $\nu \in \N^m$:
\[
b_t^{(\nu)}(x) = c(\nu,\mu) x^{-\nu}b_t(x),
\]
for some constant $c(\nu,\mu)$ depending on $\nu$ and $\mu$. One checks that
\[
|c(\nu,\mu)| \leq M_b^{|\nu|} |\nu|!,
\]
from which the lemma follows.
\end{proof}

\begin{remark}\label{remlemmabmmap}
Note $M_b$ does not depend on $t$. Moreover, since $b$ is bounded, we can bound $|b_t(x)|$ independent of $t$. It will be useful later to leave the factor $|b_t(x)|$ as it is, but since $M_b$ only depends on $b$, but not on $t$, this allows us to ``ignore the $t$ variables'', i.e., consider the case $U \subseteq (0,1)^m$, which we will always do later on.
\end{remark}


\begin{lemma}\label{lemmaprepxmap}
Suppose that $f: U \subseteq T \times (0,1)^m \to \R$, where for each $t \in T$, $U_t$ is an open subset of $(0,1)^m$, is prepared in $x$, thus $f(t,x) = b_j(t,x)F(b(t,x))$, where $b_j$ is a component function of the associated bounded monomial map $b$ of $f$ (see Definition \ref{defprepx}). Then there exist $A_f, B_f > 0$ such that for any $\nu \in \N^m$, $t \in T$ and $x \in U_t$, we have that
\[
\left| f_t^{(\nu)}(x) \right| \leq x^{-\nu} |b_{t,j}(x)| B_f A_f^{|\nu|} |\nu|!.
\]
\end{lemma}
\begin{proof}
Let $t \in T$, $x \in U_t$ and $\nu \in \N^m$. By the product rule, we have that
\[
 f_t^{(\nu)}(x) = \sum_{\nu_1 + \nu_2 = \nu} C(\nu_1,\nu_2) b_{t,j}^{(\nu_1)}(x) (F \circ b_t)^{(\nu_2)}(x),
\]
where $C(\nu_1,\nu_2)$ are positive constants, depending on $\nu_1$ and $\nu_2$ only, satisfying
\[
\sum_{\nu_1 + \nu_2 = \nu} C(\nu_1,\nu_2) \leq 2^{|\nu|}.
\]
By Remark \ref{remdefprepx}, there exist $A_F,B_F>0$ such that $F$ is $(A_F,B_F,0)$-mild. By Lemma \ref{lemmabmmap}, and because $b$ has bounded range, there are $A_b,B_b >0$ such that 
\[
|b_{t,\ell}^{(\nu')}(x)| \leq x^{-\nu'} B_b A_b^{|\nu'|} |\nu'|!,
\]
for any component function $b_{t,\ell}$ of $b_t$ and $\nu' \in \N^m$. Therefore, using the Fa\`a di Bruno formula \ref{faa} and Lemma \ref{lemmaAB}, one can find $A,B>0$, depending only on $b$ and $F$, such that
\[
\left| (F \circ b_t)^{(\nu_2)}(x) \right| \leq x^{-\nu_2} B A^{|\nu_2|} |\nu_2|!.
\] 
(You need the relation $-\sum_{j=1}^s |k_j|l_j = -\nu_2$, coming from the Fa\`a di Bruno formula, to show this.) By Lemma \ref{lemmabmmap}, we also have that
\[
\left| b_{t,j}^{(\nu_1)}(x) \right| \leq x^{-\nu_1}M_{b_j}^{|\nu_1|} |\nu_1|! |b_{j,t}(x)|.
\]
Now we conclude as follows:
\begin{align*}
\left| f_t^{(\nu)}(x) \right| &\leq \sum_{\nu_1 + \nu_2 = \nu} C(\nu_1,\nu_2) \left| b_{t,j}^{(\nu_1)}(x) \right| \left| (F \circ b_t)^{(\nu_2)}(x) \right| \\
&\leq \sum_{\nu_1+\nu_2 = \nu} C(\nu_1,\nu_2) x^{-\nu_1}M_{b_j}^{|\nu_1|} |\nu_1|! |b_{j,t}(x)| x^{-\nu_2} B A^{|\nu_2|} |\nu_2|! \\
&\leq x^{-\nu} |b_{j,t}(x)| B_f |\nu|! \sum_{\nu_1+\nu_2 = \nu} C(\nu_1,\nu_2) M_{b_j}^{|\nu_1|} B A^{|\nu_2|} \\
& \leq x^{-\nu} |b_{j,t}(x)| B_f A_f^{|\nu|} |\nu|!,
\end{align*}
where $A_f = 2\max(A,M_{b_j})$ and $B_f = B$.
\end{proof}

\begin{remark}\label{remlemmaprepxmap}
In fact the lemma shows that the statement for maps prepared in $x$ can be reduced to the same statement for bounded monomial maps. So, although one should be careful, we can safely assume $f = b_j$ for the computations we plan to do. Moreover, we see that $A_f$ and $B_f$ only depend on $f$, and thus are uniform. From now on, we will always ignore the $t$ variables for this reason.
\end{remark}

The following proposition is the power substitution result we mentioned in the beginning of this section. Essentially, it is \cite[Proposition 2.13]{smoothpara}, which is a more precise version of the original statement by Cluckers, Pila and Wilkie in \cite{unif}. It is a result on a composition of a map that is prepared in $x$ and a certain ``power function'', which will yield that the composition is mild up to order $r$. The proof strategy is to apply the Fa\`a di Bruno formula and then bound each term in this expression in such a way that we can apply Lemma \ref{lemmaAB}. This is also the original proof strategy, but, as we stated before, we will write it down in a slightly different way. From this point of view, the proof of Theorem \ref{thmcrpara} in the next section will be more clear (hopefully).

\begin{remark}\label{remconvention}
We make the following conventions with respect to properties of families of functions. If $f: U \subseteq T \times (0,1)^m \to \R^n$ is a family of functions, we say that it is $(A,B,C)$-mild up to order $r$ if for any $t \in T$, $f_t$ is $(A,B,C)$-mild up to order $r$. Moreover, to compute the $C^1$-norm of $f$, one only takes differentiation with respect to the $x$ variables into account. More precisely, by $|f|_1 \leq B$, we mean that for any $t \in T$, $\left|f_t\right|_1 \leq B$.
\end{remark}

\begin{proposition}\label{proppowersub}
Let $f: U \subseteq T \times (0,1)^m \to \R$ be prepared in $x$, where for each $t \in T$, $U_t$ is an open subset of $(0,1)^m$, and suppose that its associated bounded monomial map $b$ has bounded $C^1$-norm. Let $r \in \N$ and define the map $P_r: T \times (0,1)^m \to T \times (0,1)^m$ by
\[
P_r(t,x_1,\ldots,x_m) = (t,x_1^r,\ldots,x_m^r).
\]
Then there exist $A,B > 0$ such that the function $f \circ P_r: P_r^{-1}(U) \to \R$ is $(Ar,B,0)$-mild up to order $r$.
\end{proposition}

\begin{proof}
Arguing as the proof of Lemma \ref{lemmaprepxmap}, we may suppose that $f$ is equal to a bounded monomial map $b: C \to \R$, which has bounded $C^1$-norm. In order to show that $b \circ P_r$ is $(Ar,B,0)$-mild up to order $r$, let $t \in T$ be fixed but arbitrary. We have to show that $(b \circ P_r)_t = b_t \circ P_{r,t}$ is $(Ar,B,0)$-mild for some $A,B>0$, independent of $t$. To do so, we want to apply Lemma \ref{lemmaAB}, after using the Fa\`a di Bruno formula and suitably bounding the derivatives of $b_t$ and $P_{r,t}$. By the uniform bound on the $C^1$-norm of $b$, and by the form of $b$ and $P_r$, we may ignore the $t$ variables (see also Remark \ref{remlemmabmmap} and Remark \ref{remlemmaprepxmap}).

Let $\nu \in \N^m$ with $|\nu| \leq r$ and $x \in P_r^{-1}(U)$. By the Fa\`a di Bruno formula, Proposition \ref{faa}, and the triangle inequality, we have that
\[
\left|(b \circ P_r)^{(\nu)}(x)\right| \leq \sum_{1 \leq |\lambda| \leq |\nu|} \left| b^{(\lambda)}(P_{r}(x)) \right| \sum_{s = 1}^{|\nu|} \sum_{p_s(\nu,\lambda)} \nu! \prod_{j = 1}^s \frac{\left|\left(P_{r}^{(l_j)}(x)\right)^{k_j}\right|}{k_j!(l_j!)^{|k_j|}}.
\]
We want to find $A_1,A_2,B_1,B_2>0$, such that for a fixed term in this sum, so some $\lambda \in \N^m$, $s \in \{1,\ldots,|\nu|\}$ and $k_1,\ldots,k_s,l_1,\ldots,l_s \in p_s(\nu,\lambda)$, we have that
\begin{equation}\label{tobound}
\left| b^{(\lambda)}(P_{r}(x)) \right| \prod_{j = 1}^s \frac{\left|\left(P_{r}^{(l_j)}(x)\right)^{k_j}\right|}{k_j!(l_j!)^{|k_j|}} \leq B_1A_1^{|\lambda|}|\lambda|! \prod_{j=1}^s \frac{(B_2A_2^{|l_j|} |l_j|!)^{|k_j|}}{k_j!(l_j)!^{|k_j|}}.
\end{equation}
We will show that we can do so with $A_1 = M_b$ (as in Lemma \ref{lemmabmmap}) and $A_2 = r$. Then the proposition follows by Lemma \ref{lemmaAB}. The claim on $A_2$ is rather clear by the form of $P_{r}$ and it is clear that $B_2$ will be $1$. Let us now prove this formally.

We consider a fixed term as described above, so of the form of the left hand side of (\ref{tobound}). Let $I$ be such that $x_I = \min \{x_i \mid \lambda_i \neq 0\}$. By Lemma \ref{lemmabmmap}, we have that
\[
\left| b^{(\lambda)}(x) \right| \leq x^{-\lambda} M_b^{|\lambda|} |\lambda|! |b(x)|.
\]
Now write $\lambda = \lambda' + \beta$, where $\beta$ corresponds to the first order derivative with respect to $x_I$. We find that
\begin{equation} \label{argumentC1}
\left| b^{(\lambda)}(x) \right| = \left| \left(b^{(\beta)}\right)^{(\lambda')}(x) \right| \leq x^{-\lambda'}M_b^{|\lambda|}|\lambda|! \left|b^{(\beta)}(x)\right|. \tag{C1}
\end{equation}
By our assumption on $b$ it follows that there exists some $B_1>0$ such that $\left|b^{(\beta)}(x)\right| \leq B_1$, where $B_1$ does not depend on $x$ (in particular on the choice of $I$). Note that $M_b$ and $B_1$ are indeed independent of $t$.

Now the first part of the proof is finished by bounding the left hand side of (\ref{tobound}) as follows:
\begin{align}\label{bound2}
\left| b^{(\lambda)}(P_{r}(x)) \right| \prod_{j = 1}^s \frac{\left|\left(P_{r}^{(l_j)}(x)\right)^{k_j}\right|}{k_j!(l_j!)^{|k_j|}} &\leq P_{r}(x)^{-\lambda'} B_1 M_b^{|\lambda|} |\lambda|!\prod_{j = 1}^s \frac{\left|(P_{r}^{(l_j)}(x))\right|^{k_j}}{k_j!(l_j!)^{|k_j|}} \nonumber \\
&= P_{r,I}(x) B_1 M_b^{|\lambda|} |\lambda|! \prod_{j = 1}^s \frac{\left|\left(P_{r}^{(l_j)}(x)\right)^{k_j}\right|}{P_{r}(x)^{k_j}} \frac{1}{k_j!(l_j!)^{|k_j|}},
\end{align}
where we have used that $\lambda = \sum_{j=1}^s k_j$. Thus, by the particular form of $b$, the computations above naturally lead us to finding the desired upper bound for
\[
\frac{\left| \left(P_{r}^{(l_j)}(x)\right)^{k_j} \right|}{P_{r}(x)^{k_j}} = \prod_{\ell = 1}^m \left(\frac{\left| P_{r,\ell}^{(l_j)}(x) \right|}{P_{r,\ell}(x)}\right)^{k_{j,\ell}}.
\]
This is the second part of the proof. Obviously, it can be unbounded, but that is where the factor $P_{r,I}(x) = x_I^r$, which we obtained by the fact that $b$ has bounded $C^1$-norm, comes into play. We have that $P_{r,\ell}(x) = x_\ell^r$, thus it follows that
\[
\left(\frac{\left| P_{r,\ell}^{(l_j)}(x) \right|}{P_{r,\ell}(x)}\right)^{k_{j,\ell}} \leq (r^{|l_j|} |l_j|! x_\ell^{-l_{j,\ell}})^{k_{j,\ell}} \leq (r^{|l_j|} |l_j|!)^{k_{j,\ell}} x_I^{-k_{j,\ell}l_{j,\ell}},
\]
where the last inequality holds because $x_I \leq x_\ell$ if $\lambda_\ell \neq 0$, by the choice of $I$, and because $k_{j,\ell} = 0$ for all $j$ if $\lambda_\ell = 0$, by the relation $\sum_{j=1}^s k_j = \lambda$ in $\N^m$. We conclude that
\[
\prod_{\ell = 1}^m \left(\frac{\left| P_{r,\ell}^{(l_j)}(x) \right|}{P_{r,\ell}(x)}\right)^{k_{j,\ell}} \leq x_I^{-\sum_{\ell = 1}^m k_{j,\ell}l_{j,\ell}}(r^{|l_j|}|l_j|!)^{|k_j|}.
\]
Finally, we have that
\[
\prod_{j=1}^s x_I^{-\sum_{\ell = 1}^m k_{j,\ell}l_{j,\ell}} =  x_I^{-\sum_{\ell=1}^m \sum_{j=1}^s k_{j,\ell}l_{j,\ell}}  \leq x_I^{-|\nu|},
\]
since $\sum_{j=1}^s |k_j|l_j = \nu$. We can now further bound (\ref{bound2}):
\[
P_{r,I}(x) B_1M^{|\lambda|}|\lambda|! \prod_{j = 1}^s \frac{\left|\left(P_{r}^{(l_j)}(x)\right)^{k_j}\right|}{P_{r}(x)^{k_j}} \frac{1}{k_j!(l_j!)^{|k_j|}} \leq x_I^{r-|\nu|} B_1 M_b^{|\lambda|}|\lambda|! \prod_{j = 1}^s \frac{(r^{|l_j|}|l_j|!)^{|k_j|}}{k_j!(l_j!)^{|k_j|}}.
\]
Since $|\nu| \leq r$, we have that $x_I^{r-|\nu|} \leq 1$. We conclude that we can pick $A_1 = M_b$, $A_2 = r$ and $B_2 = 1$, as claimed before, and thus, this completes the proof.
\end{proof}

\begin{remark}
It is crucial that the $C^1$-norm of the associated bounded monomial map is bounded. For instance, consider the bounded monomial map $x^{1/2}: (0,1) \to (0,1)$. Clearly, the function $x^{r/2}: (0,1) \to (0,1)$ is not mild up to order $r$.
\end{remark}

\section{Proofs of the main theorems} \label{secmain}
In this section, we show the main results, the $C^r$-parametrization theorem (Theorem \ref{thmcrpara}), and the mild parametrization theorem (Theorem \ref{thmmildpara}). In Section \ref{secdiscuss}, I discuss some differences and similarities between the proofs in this paper and my earlier work \cite{mildpara} and \cite{smoothpara}, and the original work by Cluckers, Pila and Wilkie \cite{unif}.

The main idea of the proofs, is to combine the Pre-parametrization Theorem \ref{prepara} with Proposition \ref{proppowersub}. More precisely, Proposition \ref{proppowersub}, yields a map $f \circ P_r: P_r^{-1}(U) \to \R$, that has the same image as $f$, but we also need a map $T \times (0,1)^m \to P_r^{-1}(U)$ such that the composition with this map remains $(Ar,1,0)$-mild up to order $r$. In the proofs below, we combine this process, and construct a map $T \times (0,1)^m \to U$, such that the composition with this map has the desired mildness property. From now on, we will work with cells $\cell$ instead of the more general domains $U$ of the previous sections.

\subsection{Proof of the first main result}
Firstly, we define a map $\phi^r: T \times (0,1)^m \to \cell$, where $\cell = \{\cell_t \mid t \in T\}$ is a family of open cells in $(0,1)^m$. Denote $\alpha_i$ and $\beta_i$ for the walls bounding $x_i$ from below and above respectively. For each $i \in \{1,\ldots,m\}$, denote $\pi_i$ for the projection $T \times (0,1)^m \to T \times (0,1)^i$ onto the first $i$ coordinates, and let $\pi_0$ be the projection onto $T$. The map $\phi^r$ is the composition $\phi^{r,m} \circ \ldots \circ \phi^{r,1}$, where for each $i \in \{1,\ldots,m\}$, $\phi^{r,i}$ is the map $\pi_{i-1}(\cell) \times (0,1)^{m-i+1} \to \pi_i(\cell) \times (0,1)^{m-i}$ defined by
\[
(t,x)  \mapsto (t,x_1,\ldots,x_{i-1},\alpha_i(t,x_1,\ldots,x_{i-1})+(\beta_i-\alpha_i)(t,x_1,\ldots,x_{i-1})x_i^r,x_{i+1},\ldots,x_m).
\]

Note that if $\cell = (0,1)^m$, then $\phi^r = P_r$, where $P_r$ is the power map of Proposition \ref{proppowersub}. 

Now suppose $\cell$ is obtained as a result of the Pre-parametrization Theorem \ref{prepara}. Assuming some lemmas that show that $\phi^r$, by its construction, inherits the properties of bounded monomial maps (or maps that are prepared in $x$), we will show the first main theorem. We will show these lemmas afterwards.

\begin{theorem} \label{thmcrpara}
Let $X$ be a power-subanalytic family of subsets of $(0,1)^n$ of dimension $m$. Then there exist constants $c = c(X) \in \N$ and $A = A(X) \in \N$ with the following property. For all $r \in \N$, there exist power-subanalytic maps $f_l: T \times (0,1)^m \to X$, for $1 \leq l \leq c$, such that for any $t \in T$, we have that
\[
\bigcup_{l=1}^c \im(f_{l,t}) = X_t,
\]
and for each $l \in \{1,\ldots,c\}$, the map $f_{l,t}$ is $(Ar,1,0)$-mild up to order $r$.
\end{theorem}

\begin{proof}
Consider a map $f: \cell \to (0,1)^n$, obtained by applying the Pre-parametrization Theorem \ref{prepara} to $X$. We will show that $f \circ \phi^r$ is $(Ar,B,0)$-mild up to order $r$, for some $A,B>0$. Since the image of $f$ is contained in $(0,1)^n$, this suffices, since we can adjust $A$ to ensure that $B = 1$, if $B > 1$. Then we obtain a finite collection of $(Ar,1,0)$-mild maps up to order $r$ that parametrize $X$ as desired.

By abuse of notation, let $f$ be a component function of $f$, furthermore, arguing as in Lemma \ref{lemmaprepxmap}, it suffices to consider the case that $f$ is a bounded monomial map $b: \cell \to \R$, which has bounded $C^1$-norm (as a result of Theorem \ref{prepara}). By the form of $\phi^r$ and the since $b$ has bounded $C^1$-norm, we will ignore the $t$ variables, as we did before. 

We follow the proof of Proposition \ref{proppowersub}. Thus, after applying the Fa\`a di Bruno formula \ref{faa}, we want to find a suitable bound for
\begin{equation} \label{crmain1}
\left| b^{(\lambda)}(\phi^r(x)) \right| \prod_{j=1}^s \frac{ \left| \left( (\phi^r)^{(l_j)}(x) \right)^{k_j} \right| }{ k_j! (l_j!)^{|k_j|} }
\end{equation}
in order to apply Lemma \ref{lemmaAB}, where $x \in (0,1)^m$, $\lambda \in \N^m$, $1 \leq s \leq |\nu|$ and $k_1,\ldots,k_s$, $l_1,\ldots,l_s \in p_s(\nu,\lambda)$, coming from the Fa\`a di Bruno formula, are fixed, and where $\nu \in \N^m$ is the derivative of $b \circ \phi^r$ we are considering.

Let $1 \leq I \leq m$ be such that $x_I = \min\{ x_i \mid \nu_i \neq 0\}$. Since $\nu_I = \sum_{j=1}^s |k_j|l_{j,I}$, this implies that there exist $1 \leq j \leq s$ and $1 \leq J \leq m$ such that $k_{j,J}l_{j,I} \neq 0$. In particular, $k_{j,J} \neq 0$ and $l_{j,I} \neq 0$, which implies that $\lambda_J \neq 0$, since $\sum_{j=1}^s k_j = \lambda$. (Note that in the proof of Proposition \ref{proppowersub}, we could choose $I = J$, but that is not the case here.)

Similar to (\ref{argumentC1}) in the proof of Proposition \ref{proppowersub}, we have that
\[
\left| b^{(\lambda)}(\phi^r(x)) \right| \leq \phi^r_J(x) \phi^r(x)^{-\lambda} B_b M_b^{|\lambda|} |\lambda|!,
\]
where we have used that $b$ has bounded $C^1$-norm. We can now further bound (\ref{crmain1}) by:
\begin{equation} \label{crmain2}
B_b M_b^{|\lambda|} |\lambda|!  \phi^r_J(x) \prod_{j=1}^s \frac{ \left| \left( (\phi^r)^{(l_j)}(x) \right)^{k_j} \right| }{ \phi^r(x)^{k_j} } \frac{1}{ k_j! (l_j!)^{|k_j|} }.
\end{equation}
This finishes the first part of the proof. We are, as in Proposition \ref{proppowersub}, lead to finding an upper bound for
\[
\left| \frac{ (\phi^r_\ell)^{(l_j)}(x) }{ \phi^r_\ell(x) } \right|,
\]
where $1 \leq \ell \leq m$, which is the second part of the proof. By Lemma \ref{lemmaweaklymild}, we have that 
\[
\left| \frac{ (\phi^r_\ell)^{(l_j)}(x) }{ \phi^r_\ell(x) } \right| \leq x^{-l_j} B_\cell (A_\cell r)^{|l_j|} |l_j|!.
\]
If $\ell = J$, we have that $k_{j,J} > 0$. Instead of applying the above upper bound $k_{j,J}$ times, we will only do so $k_{j,J}-1$ times, and once use the fact that we have one factor $\phi^r_J(x)$ in front of the product in (\ref{crmain2}). Note that by the definition of $k_{j,J}$, $l_{j,I} \neq 0$. By Lemma \ref{factorxr}, we obtain that
\[
\left| \phi^r_J(x)^{(l_j)} \right| x_I^r x^{-l_j} B_\cell (A_\cell r)^{|l_j|} |l_j|!.
\]
Then we can further bound (\ref{crmain2}) as follows:
\begin{multline*}
B_b M_b^{|\lambda|} |\lambda|!  \phi^r_J(x) \prod_{j=1}^s \frac{ \left| \left( (\phi^r)^{(l_j)}(x) \right)^{k_j} \right| }{ \phi^r(x)^{k_j} } \frac{1}{ k_j! (l_j!)^{|k_j|} } \\ \leq B_b M_b^{|\lambda|} |\lambda|! x_I^r \prod_{j=1}^s \frac{ (x^{-l_j} B_\cell (A_\cell r)^{|l_j|} |l_j|!)^{|k_j|} }{ k_j! (l_j!)^{|k_j|} }.
\end{multline*}
Since $\sum_{j=1}^s |k_j|l_j = \nu$, and by our choice of $I$, we have that
\[
x_I^r x^{-\sum_{j=1}^s |k_j| l_j} \leq 1.
\]
Thus, we can conclude by Lemma \ref{lemmaAB}, since we have suitably bounded each term coming from the Fa\`a di Bruno formula.
\end{proof}

Now we show the two lemmas. Firstly, the claim on the quotient of derivatives of $\phi^r$ and $\phi^r$ itself. Note that if $\cell = (0,1)^m$, and thus the map $\phi^r$ equals the power map $P_r$ of Section \ref{secinterm}, then the statements are easy to see. The first lemma is also straightforward if one would replace $\phi^r$ by a bounded monomial map. Since $\phi^r$ is a composition of $m$ maps, we will use induction, of course.

\begin{lemma}\label{lemmaweaklymild}
Suppose that $\cell$ is a family of open cells in $(0,1)^m$ such that its walls are prepared in $x$. Then there exist $A_\cell,B_\cell > 0$ such that for any $t \in T$, $x \in (0,1)^m$ and $\nu \in \N^m$, we have that
\[
\left| \frac{ (\phi^r_{t,\ell})^{(\nu)}(x) }{  \phi^r_{t,\ell}(x) } \right| \leq x^{-\nu} B_\cell (A_\cell r)^{|\nu|} |\nu|!.
\]
for each component function $\phi^r_{t,\ell}$ of $\phi^r_{t}$.
\end{lemma}

\begin{proof}
We will prove this lemma via induction on $m$. The base case $m = 1$ is easy to verify (one might as well take $m = 0$ as a base case). Let us suppose that the result is true for values up to $m-1$ for some $m>1$. We will ignore the $t$ variables. Let $x \in (0,1)^m$ and $\nu \in \N^m$. By the induction hypothesis, it suffices to consider the case $\ell = m$. By the definition of $\phi^r$, we have that
\begin{equation} \label{eqderwallmap}
(\phi^r_m)^{(\nu)}(x) = (\alpha_m \circ  \overline{\phi^r_m})^{(\nu)}(x) + \left((\beta_m - \alpha_m) \circ \overline{\phi^r_m} \right)^{(\bar{\nu})}(x) r\cdots (r-\nu_m+1)x_m^{r-\nu_m},
\end{equation}
where $\bar{\nu} = (\nu_1,\ldots,\nu_{m-1})$ and $\overline{\phi^r_m}(x) = (\phi^r_1(x),\ldots,\phi^r_{m-1}(x))$. Therefore, we find that
\begin{align*}
\frac{ \left| (\phi^r_m)^{(\nu)}(x) \right| }{ \phi^r_m(x) } \leq &\frac{ \left|(\alpha_m \circ \overline{\phi^r_m})^{(\bar{\nu})}(x)\right| + \left|((\beta_m - \alpha_m) \circ \overline{\phi^r_m})^{(\bar{\nu})}(x) \right| r^{\nu_m} \nu_m! x_m^{r-\nu_m} }{ (\alpha_m \circ \overline{\phi^r_m})(x) + ((\beta_m - \alpha_m)\circ \overline{\phi^r_m})(x)x_m^r } \\
\leq &\frac{ \left| (\alpha_m \circ \overline{\phi^r_m})^{(\bar{\nu})}(x) \right| }{ (\alpha_m \circ \overline{\phi^r_m})(x) } + \frac{ \left| (\beta_m \circ \overline{\phi^r_m})^{(\bar{\nu})}(x) \right| x_m^r}{ (\beta_m \circ \overline{\phi^r_m})(x)x_m^r }r^{\nu_m}\nu_m!x_m^{-\nu_m} \\
&\qquad+ \frac{ \left| (\alpha_m \circ \overline{\phi^r_m})^{(\bar{\nu})}(x) \right| x_m^r}{ (\alpha_m \circ \overline{\phi^r_m})(x)x_m^r }r^{\nu_m}\nu_m!x_m^{-\nu_m} \\
\leq &\left(2\frac{ \left|\ (\alpha_m \circ \overline{\phi^r_m})^{(\bar{\nu})}(x) \right| }{ (\alpha_m \circ \overline{\phi^r_m})(x) } + \frac{ \left|\ (\beta_m \circ \overline{\phi^r_m})^{(\bar{\nu})}(x) \right| }{ (\beta_m \circ \overline{\phi^r_m})(x) } \right)r^{\nu_m}\nu_m!x_m^{-\nu_m}. \\
\end{align*}
Therefore, our statement about $\phi^r$ has been reduced to a statement about $\alpha_m$ and $\beta_m$. Note that the map $(\phi^r_1,\ldots,\phi^r_{m-1})$ is the map $\phi^r$ for the cell $\pi_{m-1}(\cell)$. Using the Fa\`a di Bruno formula \ref{faa} and lemmas \ref{lemmabmmap} and \ref{lemmaprepxmap}, one reduces the statement on $\alpha_m$ and $\beta_m$ to the map $(\phi^r_1,\ldots,\phi^r_{m-1})$, and concludes by the induction hypothesis.
\end{proof}

Finally, we show the second lemma.

\begin{lemma}\label{factorxr}
Suppose that $\cell$ is a family of open cells in $(0,1)^m$ such that its walls are prepared in $x$. Let $1 \leq I \leq m$ and suppose that $\nu \in \N^m$ such that $\nu_I \neq 0$. Then there are $A_\cell,B_\cell>0$ such that for any $t \in T$ and $x \in (0,1)^m$, we have that
\[
\left| (\phi^r_{t,\ell})^{(\nu)}(x) \right| \leq x_I^r x^{-\nu} B_\cell (A_\cell r)^{|\nu|} |\nu|!
\]
for any component function $\phi^r_{t,\ell}$ of $\phi^r_t$.
\end{lemma}
\begin{proof}
We will also prove this statement by induction on $m$. We will ignore the $t$ variables. Note that without the additional factor $x_I^r$ in the upper bound, the statement follows from the proof of Lemma \ref{lemmaweaklymild}. The claim is that if we do not divide by $\phi^r_\ell(x)$, and we know that $\nu_I \neq 0$, then we can obtain the extra factor $x_I^r$. By the induction hypothesis, we only have to consider the case $\ell = m$. Suppose first that $I = m$. As in (\ref{eqderwallmap}) of Lemma \ref{lemmaweaklymild}, we have that
\[
(\phi^r_m)^{(\nu)}(x) =\left((\beta_m - \alpha_m)\circ \overline{\phi^r_m} \right)^{(\bar{\nu})}(x)r\cdots (r-\nu_m+1)x_m^{r-\nu_m},
\]
where $\bar{\nu} = (\nu_1,\ldots,\nu_{m-1})$ and $\overline{\phi^r_m}(x) = (\phi^r_1(x),\ldots,\phi^r_{m-1}(x))$. Note that the first term vanished here since $\nu_I = \nu_m \neq 0$. Thus, we find a factor $x_I^r$, as desired. The remaining expression is bounded similarly as in the proof of Lemma \ref{lemmaweaklymild}.

Now suppose $I < m$. From the proof of Lemma \ref{lemmaweaklymild}, we deduce that
\[
\left| (\phi^r_m)^{(\nu)}(x) \right|  \leq (2\left|\ (\alpha_m \circ \overline{\phi^r_m})^{(\bar{\nu})}(x) \right| + \left|\ (\beta_m \circ \overline{\phi^r_m})^{(\bar{\nu})}(x) \right|)r^{\nu_m}\nu_m!x_m^{-\nu_m}.
\]
Since $I < m$, we have that $\bar{\nu}_I \neq 0$. Thus, our statement about $\phi^r$ is reduced to a statement about the walls of $\cell$. The same arguments as in the end of Lemma \ref{lemmaweaklymild} apply here, i.e., we can use the induction hypothesis now.
\end{proof}

\subsection{Proof of the second main result}
Let $\cell = \{\cell_t \mid t \in T\}$ be a family of open cells in $(0,1)^m$. To obtain mildness up to order $+\infty$, we will use the map $\phinf: T \times (0,1)^m \to \cell$, which is the composition of the maps $\phi^{\infty,m} \circ \ldots \phi^{\infty,1}$, where for each $i \in \{1,\ldots,m\}$, $\phi^{\infty,i}$ is the map $\pi_{i-1}(\cell) \times (0,1)^{m-i+1} \to \pi_i(\cell) \times (0,1)^{m-i}$ defined by
\[
(t,x) \mapsto (t, x_1 , \ldots, x_{i-1}, \alpha_i(t,x_1,\ldots,x_{i-1}) + (\beta_i - \alpha_i)(t,x_1,\ldots,x_{i-1}) e^{1-1/x_i^\kappa}, x_{i+1}, \ldots, x_m),
\]
where $\kappa$ is a strictly positive real number. It is shown in \cite{mildpara} that the function $e^{1-1/x^\kappa}$ is $(c(\kappa),e,1/\kappa)$-mild, where $c(\kappa) = 6\kappa$ if $\kappa \geq 1$ and $c(\kappa) = 3(2/\kappa)^{1/\kappa}$ if $\kappa \leq 1$. This function, for $\kappa = 1$, also appears in \cite{countpointsexp} and \cite{butler} in the parametrization of some particular sets in $\R^3$, and in \cite{ccs} for the same purpose as this paper.

The proof strategy of the second main theorem is exactly the same as the first main theorem. Therefore, we will keep it a bit shorter here. First, we will state the adapted version of lemmas \ref{lemmaweaklymild} and \ref{factorxr}.

\begin{lemma}\label{lemmaweaklymild2}
Suppose that $\cell$ is a family of open cells in $(0,1)^m$ such that its walls are prepared in $x$. Then there exist $A_\cell,B_\cell > 0$ such that for any $t \in T$, $x \in (0,1)^m$ and $\nu \in \N^m$, we have that
\[
\left| \frac{ (\phinf_{t,\ell})^{(\nu)}(x) }{  \phinf_{t,\ell}(x) } \right| \leq x^{-(\kappa+1)\nu} B_\cell (c(\kappa) A_\cell)^{|\nu|} |\nu|!.
\]
for each component function $\phinf_{t,\ell}$ of $\phinf_{t}$, where $c(\kappa) = \kappa$ if $\kappa \geq 1$, and $c(\kappa) = 1$ if $0 <\kappa \leq 1$.
\end{lemma}
\begin{proof}
Using the Fa\`a di Bruno formula \ref{faa} to the composition of $e^x$ with $1-1/x^\kappa$ (see also the proof of \cite[Proposition 3.4]{mildpara}), one finds that that for any $x \in (0,1)$ and $\nu \in \N$:
\begin{align*}
\left| \left( e^{1-1/x} \right)^{(\nu)} \right| &\leq x^{-(\kappa+1)\nu} \sum_{1 \leq \lambda \leq \nu} e^{1-1/x^\kappa} \sum_{s=1}^\nu \sum_{p_s(\nu,\lambda)} \nu! \prod_{j=1}^s \frac{ (c(\kappa)^{l_j} l_j!)^{k_j} }{ k_j! (l_j)!^{k_j} } \\
&\leq x^{-(\kappa+1)\nu} e^{1-1/x^\kappa} \sum_{1 \leq \lambda \leq \nu} \sum_{s=1}^\nu \sum_{p_s(\nu,\lambda)} \nu! \prod_{j=1}^s \frac{ (c(\kappa)^{l_j} l_j!)^{k_j} }{ k_j! (l_j)!^{k_j} },
\end{align*}
where $c(\kappa)$ is $\kappa$ if $\kappa \geq 1$ and is $1$ if $\kappa \leq 1$ (see \cite[Lemma 3.2]{mildpara}) . Applying Lemma \ref{lemmaAB}, we obtain
\[
\left| \left( e^{1-1/x} \right)^{(\nu)} \right| \leq x^{-(\kappa+1)\nu} e^{1-1/x^\kappa} (2c(\kappa))^{\nu} \nu!.
\]
Using this upper bound, one repeats the arguments of Lemma \ref{lemmaweaklymild} to obtain the result.
\end{proof}

We see that (obviously) the function $e^{1-1/x^\kappa}$ appears in the upper bound on its derivatives. This observation, just as in Lemma \ref{factorxr}, leads to the second lemma. The proof is also similar, so we omit it.

\begin{lemma}\label{factorexp}
Suppose that $\cell$ is a family of open cells in $(0,1)^m$ such that its walls are prepared in $x$. Let $1 \leq I \leq m$ and suppose that $\nu \in \N^m$ such that $\nu_I \neq 0$. Then there are $A_\cell,B_\cell>0$ such that for any $t \in T$ and $x \in (0,1)^m$, we have that
\[
\left| (\phinf_{t,\ell})^{(\nu)}(x) \right| \leq e^{1-1/x_I^\kappa} x^{-(\kappa+1)\nu} B_\cell (c(\kappa) A_\cell)^{|\nu|} |\nu|!
\]
for any component function $\phinf_{t,\ell}$ of $\phinf_t$, where $c(\kappa) = \kappa$ if $\kappa \geq 1$, and $c(\kappa) = 1$ if $0 <\kappa \leq 1$.
\end{lemma}

We can now prove the second main theorem.

\begin{theorem} \label{thmmildpara}
Let $X$ be a power-subanalytic family of subsets of $(0,1)^n$ of dimension $m$. Then there exist constants $c = c(X) \in \N$ and $A = A(X) \in \N$ with the following property. For all $C>1$, there exist maps $f_l: T \times (0,1)^m \to X$, for $1 \leq l \leq c$, such that for any $t \in T$, we have
\[
\bigcup_{l=1}^c \im(f_{l,t}) = X_t,
\]
and for each $l \in \{1,\ldots,c\}$, the map $f_{l,t}$ is $(A/C,1,C)$-mild up to order $+\infty$.
\end{theorem}
\begin{proof}
We use the same notation and strategy as in the proof of Theorem \ref{thmcrpara}. Doing so, using lemmas \ref{lemmaweaklymild2} and \ref{factorexp} instead, we find the following upper bound for a fixed term in the sum obtained by estimating the Fa\`a di Bruno formula:
\begin{multline*}
\left| b^{(\lambda)}(\phinf(x)) \right| \prod_{j=1}^s \frac{ \left| \left( (\phinf)^{(l_j)}(x) \right)^{k_j} \right| }{ k_j! (l_j!)^{|k_j|} } \\ \leq B_b M_b^{|\lambda|} |\lambda|! e^{1-1/x_I^\kappa} \prod_{j=1}^s \frac{ (x^{-(\kappa+1)l_j} B_C (c(\kappa) A_C)^{|l_j|} |l_j|!)^{|k_j|} }{ k_j! (l_j!)^{|k_j|} }.
\end{multline*}
Since $\sum_{j=1}^s |k_j|l_j = \nu$, and by the choice of $I$, we have that
\[
\prod_{j=1}^s \left( x^{-(\kappa+1)l_j} \right)^{|k_j|} \leq x_I^{-(\kappa+1)|\nu|}.
\]
One computes that for any $x_I \in (0,1)$:
\[
e^{1-1/x_I^\kappa} x_I^{-(\kappa+1)|\nu|} \leq e \left( \frac{(\kappa+1)|\nu|}{e\kappa} \right)^{(\kappa+1)|\nu|/\kappa} \leq e \left( \frac{\kappa+1}{\kappa} \right)^{(\kappa+1)|\nu|/\kappa} |\nu|!^{1+1/\kappa}.
\]
(See, for instance, \cite[Lemma 3.3]{mildpara}.) This factor is the same for each fixed term and thus can be moved out of the summation. By Lemma \ref{lemmaAB}, we find $\tilde{A},B>0$, such that
\[
\left| (b \circ \phinf)^{(\nu)}(x) \right| \leq \left( \frac{\kappa+1}{\kappa} \right)^{(\kappa+1)|\nu|/\kappa} |\nu|!^{1+1/\kappa} B (c(\kappa)\tilde{A})^{|\nu|} |\nu|!.
\]
Therefore, the composition is $(A,B,1+1/\kappa)$-mild up to order $+\infty$, where $A = 4\kappa\tilde{A}$ if $\kappa \geq 1$ and $A = \left(\frac{\kappa+1}{\kappa} \right)^{(\kappa+1)/\kappa} \tilde{A}$ if $0 < \kappa \leq 1$.
\end{proof}

\subsection{Remarks about the proofs} \label{secdiscuss}
In this section I discuss some differences and similarities between the techniques in this paper and my earlier work \cite{mildpara} and \cite{smoothpara}, and the original results obtained by Cluckers, Pila and Wilkie in \cite{unif}.

\begin{remark}
In both methods, one first constructs a ``good'' $C^1$-parametrization of $X$, which is in this case the Pre-parametrization Theorem \ref{prepara}. Next, one has to ensure that the higher order derivatives are bounded, by applying Proposition \ref{proppowersub} and reparametrize the domain such that it is $(0,1)^m$. In the proofs above, we simultaneously bound the higher order derivatives and reparametrize the domain. This simultaneous construction generalizes Proposition \ref{proppowersub}, and is the reason why we now achieve $cr^m$ charts in the $C^r$-parametrization theorem.

Now let me discuss the original approach by Cluckers, Pila and Wilkie in \cite{unif}. It is desirable to just apply Proposition \ref{proppowersub}, but then one encounters the following problem, which I will explain by an example. Consider the function $f: \cell \to (0,1)$ defined by $f(x_1,x_2) = x_1^3/x_2$, where $\cell = \{ (x_1,x_2) \in (0,1)^2 \mid x_1^{3/2} < x_2 < 1\}$. In particular, this function could potentially be an outcome of the Pre-parametrization Theorem \ref{prepara}. Now, if one applies Proposition \ref{proppowersub}, we obtain the map $f \circ P_r: P_r^{-1}(\cell) \to (0,1)$, where $P_r^{-1}(\cell) =\cell$ in this case, which is $(Ar,B,0)$-mild up to order $r$ for some $A,B>0$, as desired. However, if one maps $(0,1)^2$ onto $P_r^{-1}(\cell)$ using the map $\phi$ defined by
\[
(x_1,x_2) \mapsto (x_1, x_1^{3/2} + (1-x_1^{3/2})x_2),
\]
one sees that $\phi$ not mild up to order $r$, if $r > 1$, so neither is the composition with $f \circ P_r$. This problem is solved in \cite{unif} as follows. Since the walls of the cell $\cell$ are of the same form as the function $f$ defined on it, one can also apply Proposition \ref{proppowersub} to them. More precisely, consider the map $P_r^1: (0,1)^2 \to (0,1)^2$ defined by $P_r^1(x_1,x_2) = (x_1^r,x_2)$. Then $f \circ P_r \circ P_r^1: (P_r^1)^{-1}(\cell) \to (0,1)$ is $(Ar^2,B,0)$-mild up to order $r$, for some $A,B>0$ by Proposition \ref{propmildop}. One verifies that $(P_r^1)^{-1}(\cell) = \{ (x_1,x_2) \in (0,1)^2 \mid x_1^{3r/2} < x_2 < 1\}$. We see that in this way, Proposition \ref{proppowersub} is applied to the wall $x_1^{3/2}$ of $\cell$. Now the walls of this cell are $(Ar,B,0)$-mild up to order $r$ for some $A,B>0$, hence so is the map $(0,1)^2 \to (P_r^1)^{-1}(\cell)$. Note that if the walls are prepared in $x$, they can slightly change doing these operations. The map $P_r \circ P_r^1: (0,1)^2 \to (0,1)^2$ is given by $(x_1,x_2) \mapsto (x_1^{r^2},x_2^r)$. This explains why in \cite[Proposition 4.2.6]{unif}, they use a modified version of Proposition \ref{proppowersub} that uses instead of the map $P_r$, the map $P: (0,1)^m \to (0,1)^m$ defined by
\[
(x_1,\ldots,x_m) \mapsto (x_1^{r^m}, x_2^{r^{m-1}},\ldots,x_m^r).
\]
In this way, they solved this problem and were able to establish a $C^r$-parametrization theorem where the number of maps in polynomial in $r$. In \cite{smoothpara}, I carefully analyzed their work, and combining with Proposition \ref{propmildop}, which was perhaps not known to the authors of \cite{unif}, it follows that this method yields $cr^{m^3}$ charts.
\end{remark}

\begin{remark}
It is crucial in the approach in \cite{unif} that the walls of the cells are prepared in $x$ and that their associated bounded monomial maps have bounded $C^1$-norms (such that one can apply a variation on Proposition \ref{proppowersub}, as described above). This is not the case for the proof of the $C^r$-parametrization theorem in this paper. However, this does not add much generality, since in the proof of the Pre-parametrization Theorem, it is in general required to bound the $C^1$-norm of the associated bounded monomial maps of the walls of $\cell$ to ensure that the associated bounded monomial map of the function defined on $\cell$ also has bounded $C^1$-norm.
\end{remark}

\begin{remark}
The proof of the $C$-mild Parametrization Theorem somewhat resembles the proof of \cite[Proposition 4.3]{mildpara}, which also yields maps with the same mildness as the maps constructed in the $C$-mild Parametrization Theorem. Now in \cite{mildpara}, I prove a $C$-mild Parametrization Theorem for curves, where $C>0$, see \cite[Theorem 4.7]{mildpara}. This improvement seems to be restricted to the case of curves only. More precisely, the proof strategy starts again with the Pre-parametrization Theorem, applied to some curve in $(0,1)^n$. Let $f: \cell \to (0,1)^n$ be one of the maps of this parametrization, where $\cell$ is an open cell in $(0,1)$, hence $\cell = (a,b)$. Now the strategy applied here, is to consider the composition with $\phinf: (0,1) \to \cell$, which is in this case defined by
\[
x \mapsto a+(b-a)e^{1-1/x^\kappa}.
\]
Consider this map as the composition of the map $(0,1) \to \cell$ defined by $x \mapsto a + (b-a)x$ and the map $e^{1-1/x}: (0,1) \to (0,1)$. Now in the one dimensional case, one can turn around the order of these operations to obtain a map $(0,1) \to \cell$ of the form
\[
e^{1-\frac{1}{(a' + (b'-a')x)^\kappa}},
\]
which has a better interaction with bounded monomial maps. This better interaction is the key to a sharper result in \cite{mildpara}. Note that even in dimension one, I do not know if one can show that one may take $C>0$ in the $C$-mild Parametrization Theorem in this paper by simply improving the proof. If one tries to generalize the sharper version to higher dimensions, one will face some obstructions due to the inverse of the map $e^{1-1/x^\kappa}$ that will occur. Firstly, this map is not definable in $\RanR$, although this is not necessarily a problem, and secondly, its interaction with the walls is not as good as the in the case of the function $x \to x^r$, whose inverse is just $x^{1/r}$, which interacts nicely with monomials. To me, the construction in this paper seems more natural, so perhaps it is not possible to improve the sharper result of \cite{mildpara} to higher dimensions.
\end{remark}

\subsection{A diophantine application} \label{secappl}
The motivation to study parametrizations, at least in diophantine geometry, is to apply the determinant method developed by Bombieri and Pila (see \cite{determinant}). Coupled with a parametrization result, Pila and Wilkie proved the well known Counting Theorem in \cite{countthm}, which has been successfully used by Pila to prove the Andr\'e-Oort Conjecture for products of modular curves \cite{andreeoort}. See \cite{scanlon} and \cite{diogeo} for surveys on this subject.

For $q = a/b \in \Q$, where $\gcd(a,b) = 1$, we denote $H(q) = \max(|a|,|b|)$ for the height of the rational number $q$. For a tuple $q \in \Q^n$, we define the height the be the maximum of the heights of its coordinates. One can then approximate a set $X \subseteq \R^n$ using finite sets of rational points by imposing a threshold, often also denoted $H \in \N$, on the height:
\[
X(\Q,H) = \{ x \in X \cap \Q^n \mid H(q) \leq H\}.
\]
The idea is that the growth of the cardinality of the sets $X(\Q,H)$, for $H \to \infty$, gives us information about $X$. The Counting Theorem is a result in this fashion. Let $X^{\text{alg}}$ be the union of all semi-algebraic curves contained in $X$, and let $X^{\text{trans}} = X \setminus X^{\text{alg}}$.

\begin{theorem}[Counting Theorem, \cite{countthm}] \label{thmcount}
Suppose that $X$ is definable in an o-minimal expansion of the real field. Then for any $\epsilon > 0$, there exists a constant $c = c(\epsilon)$, such that for any $H \in \N$:
\[
\left| X^{\text{trans}}(\Q,H) \right| \leq cH^\epsilon.
\]
\end{theorem}

It is also conjectured in \cite{countthm} that a sharper upper bound holds if the set $X$ is definable in the particular o-minimal structure $\R_{\exp}$. This conjecture is known as Wilkie's Conjecture.

\begin{conj*}[Wilkie]
Suppose that $X$ is definable in $\R_{\exp}$. Then there exist constants $c = c(X)$ and $d = d(X)$, such that for any $H \in \N$ with $H >e$:
\[
\left| X^{\text{trans}}(\Q,H) \right| \leq c \log(H)^d.
\]
\end{conj*}

The following result of \cite{unif} could be seen as a step towards this conjecture. For $x \in \R$, denote $[x]$ for the unique integer such that $[x] \leq x < [x]+1$.

\begin{theorem}[ {\cite[Theorem 2.3.1]{unif}} ] \label{thmratpoints}
Suppose that $X$ is a power-subanalytic family of subsets of $(0,1)^n$ of dimension $m$. Then there exist constants $c_1 = c_1(X)$, $c_2 = c_2(X)$, such that for any $t \in T$ and $H \in \N$ with $H > e$, $X_t(\Q,H)$ is contained in the union of at most $$c_1 \log(H)^{c_2}$$ algebraic hypersurfaces of degree at most $$\left[\log(H)^{m/(n-m)}\right].$$
\end{theorem}

Their proof shows that if the number of maps of a $C^r$-parametrization is polynomial in $r$, then it is possible to deduce the polynomial dependence on $\log(H)$ in the theorem above. Since in this paper, the polynomial dependence is explicit, it follows immediately that this is also the case for the number of algebraic hypersurfaces. More precisely, we obtain the following.

\begin{corollary} \label{corratpoints}
In Theorem \ref{thmratpoints}, we may take $c_2(X) = 2\frac{mn}{n-m}$.
\end{corollary}
\begin{proof}
Firstly, in \cite{unif}, one uses the $C^r$-supremum norm. Therefore, we obtain (up to some constant depending on $X$) $r^{2m}$ maps instead of $r^m$ (see Remark \ref{remamountcharts}). One then follows the proof of \cite[Theorem 2.3.1]{unif} with (in their notation) $c = 2m$ (see the last line on p. 9 there). Then one finds in the end of the proof (on p. 12) that $C(b+1)^c$ algebraic hypersurfaces are required. Asymptotically (and up to some constant depending on $m$ and $n$), one has $b+1 = \log(H)^{n/(n-m)}$. Since we had that $c = 2m$, the claim follows.
\end{proof}

\begin{remark}
In \cite[Theorem 2.3.2]{unif} one has $c_2(X) = 0$ for subanalytic sets. It would be interesting to know if this can also be achieved for power-subanalytic sets.
\end{remark}

\bibliographystyle{alpha}
\bibliography{subanalyticlib}

\end{document}